\documentclass{article}
\usepackage[utf8]{inputenc}
\usepackage[english]{babel}

\usepackage{amsmath, amsthm, amssymb}
\usepackage{hyperref}
\usepackage{graphicx}
\usepackage[colorinlistoftodos]{todonotes}
\usepackage{mathtools}
\usepackage{pst-eucl}
\usepackage{wrapfig}
\usepackage{caption}
\usepackage{lipsum}
\usepackage{advdate}
\usepackage{gensymb}
\usepackage{calrsfs}
\usepackage[mathscr]{euscript}
\usepackage{indentfirst}

\DeclarePairedDelimiter{\biggg}{\bigg(}{\bigg)}

\newcommand{\rpm}{\sbox0{$1$}\sbox2{$\scriptstyle\pm$}
  \raise\dimexpr(\ht0-\ht2)/2\relax\box2 }
  
\theoremstyle{plain}
\newtheorem{theorem}{Theorem}[section]

\newtheorem{lemma}[theorem]{Lemma}
\newtheorem{conjecture}{Conjecture}[section]

\theoremstyle{definition}

\title{Establishing Conditions on the Degree of Regularity of Linear Homogeneous Equations}

\author{Nathan Johns}

\date{\AdvanceDate[-1]\today}

\begin{document}

\maketitle

\begin{abstract}
In 1933, Rado conjectured that for any positive integer $n$, there is always a linear homogeneous equation with degree of regularity $n.$ In proving this conjecture, Alexeev and Tsimerman, and independently Golowich, found that some equations in $n$ variables have degree of regularity $n-1$ for any value of $n.$ Their work left many questions as to how and which other properties of equations are closely tied to the degree of regularity, and if there is a simpler or more effective way of thinking about it. In this paper, we answer some of these questions, prove that various families of linear homogeneous equations in $n$ variables have degree of regularity $n-1$, and establish some conditions under which this property holds.
\end{abstract}
 
\section{Introduction}

Since 1927, Van der Waerden's theorem [1], which concerns assigning ``colors" to each positive integer, has been a fundamental result used in the study of Ramsey Theory and combinatorics. A coloring with $r$ colors, or an \textit{$r$-coloring}, of the positive integers is any function $C: \mathbb{N} \rightarrow [1, r]$, where $[1,r]$ denotes the set of positive integers from 1 to $r$. Hence, each of the first $r$ positive integers are considered the ``colors" of our coloring.

\begin{theorem}[Van der Waerden's Theorem]
Given any positive integers $r,k$, there is a positive integer $n_0$ such that, for all $n \geq n_0$, any $r$-coloring of $[1,n]$ produces a monochromatic arithmetic progression of length $k.$
\end{theorem}

In 1933, Richard Rado [2] further developed Van der Waerden's findings, applying them to the partition regularity of systems of linear
homogeneous equations, which concerns the properties of the cells of partitions of the solution sets of such equations. We consider a linear homogeneous equation to be any equation in $n$ variables of the form
\[a_1x_1 + ... + a_nx_n = 0 \label{eq1} \tag{1}\]
with $a_i \in \mathbb{Z} \backslash \{0\}$. If, given any coloring of the integers and a linear homogeneous equation, $S$, of the form \eqref{eq1}, we can always find $n$ monochromatic integers $x_1, x_2, . . . , x_n$ that satisfy $S$, then we say $S$ is \textit{regular}, sometimes referred to as \textit{partition regular.}  
\begin{theorem}[Rado's Theorem]

A linear homogeneous equation $a_1x_1 + ... + a_nx_n = 0$ is regular if and only if there exists a nonempty subset of coefficients $A = \{a_i, ..., a_j\}$ such that $\sum_{a_i \in A} a_i = 0$. 
\end{theorem}
In the case that a linear homogeneous equation is not regular, we can classify it further using a generalized notion of partition regularity, $r$\textit{-regularity}, defined as follows: 
for any positive integer $r$ and linear homogeneous equation $S$, $S$ is
\textit{r-regular} if and only if, for every coloring of the positive integers with exactly $r$ colors, there is a
monochromatic solution $x_1, x_2, . . . , x_n$ to the equation. Using the notion of $r$-regularity, we can classify non-regular linear homogeneous equations by their \textit{degree of regularity}, the greatest positive integer $a$ such that the equation
is $a$-regular. 

This all led to the following conjecture made by Rado [2]:
\begin{conjecture}
For every positive integer $n$, there is always at least $1$ linear
homogeneous equation with degree of regularity $n.$
\end{conjecture}

The conjecture received very little attention in the years following the publication of Rado's work, but thanks to a recent resurgence in the study of partition regularity, it was finally proven in 2009 by Alexeev and Tsimerman [3]. Since 2009, both Alexeev and Tsimerman, and Golowich [4], independently, have found distinct families of equations in $n$ variables with degree of regularity $n-1$, thus verifying Rado's conjecture. However, recent results have prompted various new questions related to the degree of regularity of systems of equations. For example, under what conditions does a linear homogeneous equation in $n$ variables have degree of regularity equal to $n-1$? 

In addition to finding more general systems of equations that verify Rado's conjecture, we will prove 2 sufficient conditions for any linear homogeneous equation: the first to not be $n$-regular, and the second to be $(n-1)$-regular.
\section{Important Lemmas}

For the following results, we must understand the notion of a homogeneous family of subsets of $\mathbb{N}$. We define a homogeneous family of subsets of $\mathbb{N}$  as a family $S \subset \mathcal{P}(\mathbb{N})$ such that for all $s \in S$ and k $\in \mathbb{N},$ $ks \in S$. In other words, if a number is in $S$ then every positive integer multiple of it is also in $S$. If we consider the solution set of any linear homogeneous equation as a family of subsets of $\mathbb{N}$, it can be seen quite easily that such a family fits the above definition of homogeneous. The following result was proven by Golowich in 2014 [4]:
\begin{lemma}
Let $S$ be a homogeneous family of subsets of $\mathbb{N}$ which is $r$-regular. Let $q$ and $M$ be positive integers. Given any $r$-coloring of the positive integers, there is a subset of the positive integers $B \in S$ and a positive integer $d$ such that, for every positive integer $b \in B$ and $l \in \mathbb{Z}$ such that $|l| \leq M$, 
$b+ld$ and $qd$ are colored with the same color.
\end{lemma}

The result has been shown to be a very useful tool in proving the $(n-1)$-regularity of equations in $n$ variables, as one can always guarantee that, given an $(n-1)$-coloring, some pair of coefficients $(a_i, a_j)$ must be monochromatic. While we will use Lemma 2.1 to prove the $(n-1)$-regularity of specific families of equations, we will also provide the proof of a similar result that will be key in our determining a sufficient condition for $(n-1)$-regularity. 

The proof of the following result uses the same method used by Golowich [4] in the proof of Lemma 2.1:

\begin{lemma} \label{lemma1}
Let $q$ and $M$ be any positive integers. Given any $r$-coloring of the positive integers, there exist $k, d > 0$ such that the arithmetic progression defined by 
\[k + \lambda d \ \ : \ \ \lambda \in \mathbb{Z}, |\lambda| \leq M  \tag{2} \label{eq2}\]
and
\[q + d \tag{3} \label{eq3}\]
are colored with the same color.
\end{lemma}
\begin{proof}
Let $q$ and $M$ be positive integers. We will prove Lemma 2.2 by induction. Let $r$ be the number of colors in our coloring. The case $r=1$ is clearly trivial, since there is only one color. Letting $s \in \mathbb{N}$ be given, assume that Lemma 2.2 holds for $r=s$; that there is a positive integer $T$ so that if we color the integers $[1,T]$ with $s$ colors, there are $k,d > 0$ satisfying \eqref{eq2} and \eqref{eq3}. By Theorem 1.1, given $s+1, 2TM+1$, there exists a positive integer $E_0$ such that, for all $E \geq E_0$, any $(s+1)$-coloring of the integers $[1,E]$ produces a monochromatic arithmetic progression of length $2TM + 1.$ Thus we have the following monochromatic arithmetic progression, letting $k,d \in \mathbb{N}$, where $k$ is the center of the progression:
\[k + \lambda d' \ \ : \ \ \lambda \in \mathbb{Z}, |\lambda| \leq TM \tag{4} \label{eq4}\]
If there is an integer $c \leq T$ where $q + cd'$ is the same color as in \eqref{eq4}, then we let $d = cd'$, guaranteeing that \eqref{eq2} and \eqref{eq3} are satisfied. If there is no such $c,$ then all of the $T$ positive integers $q + d', q + 2d', ..., q + Td'$ are colored using the remaining $s$ colors. We know, then, by the induction hypothesis, that \eqref{eq2} and \eqref{eq3} must be satisfied. By the principle of induction, Lemma 2.2 holds for any $r$-coloring of the positive integers.
\end{proof}

\section{Results}
Our first result provides a sufficient condition for a linear homogeneous equation in $n$ variables to not be $n$-regular using the notion of $p$-adic order for a prime $p$. The idea of using $p$-adic valuations in the context of colorings was first used by Fox and Radoicic [5]. 
We present $p$-adic order through a family of functions $\{O_p: \mathbb{Z} \rightarrow \mathbb{N}\}$, defined below: 

\[   O_p(x) = \left\{\begin{array}{ll}\textnormal{max}\{v \in \mathbb{N}: p^v \, | \, x\} ~ \textnormal{ if } n \neq 0 \\ \infty ~~~~~~~~~~~~~~~~~~~~~ \, \kern 0.497em  \textnormal{ if } n = 0 \end{array} \right. ,\]
the $p$-adic order of $x$, where $p$ is a prime positive integer.  

Recall the following 2 important properties of $p$-adic order, given a prime $p$ and integers $a$ and $b$: \\ \\
$1. ~ O_p(ab) = O_p(a) + O_p(b)$ \\
$2. ~ O_p(a+b) \geq $ min$\{O_p(a),O_p(b)\}$, where equality occurs when $a \neq b$. 

Using $p$-adic order, we find an $n$-coloring that does not allow for a monochromatic solution if the equation satisfies a certain condition. 
\begin{theorem}
For any linear homogeneous equation

\[a_1x_1 + ... + a_nx_n = 0 \]
if there is any prime positive integer $p$ such that $O_p(a_i) \equiv O_p(a_j)$ mod $n$ implies $i = j$ for all $1 \leq i,j \leq n$, the equation is not $n$-regular.
\end{theorem}

\begin{proof}
We proceed by showing that there is a $n$-coloring such that such an equation has no monochromatic solution, no matter the value of $n.$ Let $E_n$ be a linear homogeneous equation, and let $p$ be a positive prime integer such that $O_p(a_i) \equiv O_p(a_j)$ mod $n$ implies $i = j$ for all $1 \leq i,j \leq n.$ In other words, all $O_p(a_i)$ are distinct mod $n$. We claim that the coloring $C_n: \mathbb{N} \rightarrow [1,r]$ defined by $(r) = O_p(r)$ mod $n$ has no monochromatic solution to $E_n$. 

For the sake of contradiction, let us suppose that the coloring $C_n$ allows for a monochromatic solution to $E_n.$ For $x_1, ..., x_n$ to be the same color, we must have, for some positive integer $m < n$, that $O_p(x_i) \equiv m$ mod $n$ for all $i \leq n.$ Thus, we let $O_p(x_i) = m + k_in$ for all $i \leq  n,$ where $k_i$ are positive integers. By property 1, then, the $p$-adic order of each term of $E_n$ is given by $O_p(a_ix_i) = m + k_in + O_p(a_i)$. 

Thus, given any $i,j \leq n$, for $a_ix_i$ and $a_jx_j$ to have the same $p$-adic order, we must have that
\[k_in + O_p(a_i) = k_jn + O_p(a_j),\]
which would mean that 
\[O_p(a_i) \equiv O_p(a_j) \textnormal{ mod } n\]
By contradiction, each $O_p(a_ix_i)$ is unique. It follows, then, by Property 2, that the $p$-adic order of the left-hand side of $E_n$ equals min$\{O_p(a_1x_1), ..., O_p(a_nx_n)\}$, which must be a nonnegative integer. However, by definition, $O_p(0) = \infty$. 

By contradiction, there is no monochromatic solution to $E_n$ under $C_n$.
\end{proof} 

Theorem 3.1 provides us with a sufficient condition for an equation in $n$ variables not to be $n$-regular. As such, we will use it to verify that certain equations we can show to be $(n-1)$-regular are indeed of degree of regularity $(n-1).$ \\

\begin{theorem}
For all integers $n \geq 2$, the equations 
\[x_1 - 2x_2 + ... \ \rpm \ 2^{n-1}x_n = 0 \label{eq5} \tag{5}\]
are $(n-1)$-regular but not $n$-regular.
\end{theorem}

\begin{proof}
First we will show that the equation is always $(n-1)$-regular. If we take any $(n-1)$-coloring $C : \mathbb{N} \rightarrow [1,n-1] $, then there must be an $i,j$ in the set of $n$ integers $\{2^0, 2^1, ..., 2^{n-1}\}$ such that $C(2^i) = C(2^j)$ by the pigeonhole principle. Thus, there must be positive integers $x \in [1, 2^{n-1}]$ and $y \in [1, n-1]$, such that $C(x) = C(2^yx)$. So we define a set $S$ as follows:
\[S = \{(a,b) \in \mathbb{N}^2 \ \ : \ \ a = 2^jb \textnormal{ where } j \in [1, n-1] \}.\]
Clearly, $S$ is homogeneous and, as we have just shown, $(n-1)$-regular. Let $q = 2^{n-1}$ and $M = 2^n$.

By Lemma 2.1, there is a  positive integer $d \in \mathbb{N}$ and set $\{2^jb,b\} \in S$, where $b \in [1, n-1]$, such that
\[2^jb + \lambda_1d \ \ \kern 0.08em : \ |\lambda_1| \leq 2^n\]
\[b + \lambda_2d \ \ \ \ \ : \ |\lambda_2| \leq 2^n\]
\[2^{n-1}d\]
are the same color.  \\ \\
We first consider the case that $n$ is even. So we are dealing with the equation
\[x_1 - 2x_2 + ... - 2^{n-1}x_n = 0 \tag{6} \label{eq6}\]
Now, if $j$ is odd, then $a_{n-j}$ is positive, so we parametrize $x_1, ..., x_n$ as follows:
\[   x_i = \left\{ \begin{array}{ll} 2^{n-1}d \ \ \ \kern 0.21em  \ \ \ : \ 1 \leq i < n, i \neq n-j \\ 2^jb + \lambda_1d \ \ : \ i = n-j \\ b + \lambda_2d \ \ \ \kern 0.25em \ : \ i = n \end{array} \right.\]
where $\lambda_1$ and $\lambda_2$ are integers with $|\lambda_1|, |\lambda_2| \leq M$.
For these values to satisfy \eqref{eq6}, we must have
\[2^{n-1}d (1 - 2 + ... + 2^{n-2} - 2^{n-1-j}) + 2^{n-1-j}(2^jb + \lambda_1d) - 2^{n-1}(b + \lambda_2d) = 0 \tag{7} \label{eq7}\]
If $j$ is even, then $a_{n-j}$ is negative, the same sign as $a_n$, so we parametrize $x_1, ..., x_n$ as follows:
\[   x_i = \left\{ \begin{array}{ll} 2^{n-1}d \ \ \ \kern 0.2em  \ \ \ : \ 1 \leq i < n, i \neq n-j, i \neq n-j+1 \\ 2^jb + \lambda_1d \ \ : \ i = n-j, i = n-j+1 \\ b + \lambda_2d \ \ \ \ \kern 0.29em : \ i = n \end{array} \right.\]
Since $-2^{n-1-j} + 2^{n-j} = 2^{n-1-j},$ we are left with \eqref{eq7} again.
The expression $1 - 2 + ... + 2^{n-2} - 2^{n-1-j}$ simplifies to $\frac{1-2^{n-2}}{3} + 2^{n-2} - 2^{n-1-j},$ which simplifies to 
\[\frac{2^{n-1} + 1}{3} - 2^{n-1-j}\]
Since $d > 0$, we are left with 
\[2^{n-1}\biggg{\frac{2^{n-1} + 1}{3} - 2^{n-1-j}} + 2^{n-1-j}\lambda_1 - 2^{n-1}\lambda_2 = 0\]
Rearranging and dividing out by $2^{n-1-j},$
\[\lambda_1 - 2^j\lambda_2 = 2^{n-1} - \frac{2^{n-1+j} + 2^j}{3}\]
We may choose $\lambda_1 = 2^{n-1} \leq M$ and $\lambda_2 = \frac{2^{n-1} + 1}{3} \leq M.$ Furthermore, since $n-1$ is odd, $2^{n-1} \equiv 2 $ mod 3, so $\lambda_2$ is indeed a positive integer. Therefore, these values of $\lambda_1$ and $\lambda_2$ produce a monochromatic solution $x_1, ..., x_n$ to \eqref{eq5} if $n$ is even. \\ \\
Finally we consider the case that n is odd, giving the equation
\[x_1 - x_2 + ... + 2^{n-1}x_n = 0 \tag{8} \label{eq8}\]
If $j$ is odd, we parametrize $x_1, ..., x_n$ as follows:
\[   x_i = \left\{ \begin{array}{ll} 2^{n-1}d \ \ \ \  \ \ \kern 0.21em : \ 1 \leq i < n, i \neq n-j \\ 2^jb + \lambda_1d \ \ : \ i = n-j \\ b + \lambda_2d \ \ \kern 0.25em \ \ : \ i = n \end{array} \right.\]
If $j$ is even, we alter the parametrization identically as before to account for the sign of $a_{n-j}$:
\[   x_i = \left\{ \begin{array}{ll} 2^{n-1}d \ \kern 0.2em \ \  \ \ \ : \ 1 \leq i < n, i \neq n-j, i \neq n-j+1 \\ 2^jb + \lambda_1d \ \ : \ i = n-j, i = n-j+1 \\ b + \lambda_2d \ \ \ \ \kern 0.28em : \ i = n \end{array} \right.\]
If the above values are to satisfy \eqref{eq8}, we must have:
\[2^{n-1}(-1 - 2^2 - 2^4 - ... - 2^{n-3} + 2^{n-1-j}) - 2^{n-1-j}\lambda_1 + 2^{n-1}\lambda_2 = 0\]
Simplifying, we get
\[2^{n-1}\biggg{\frac{1-2^{n-1}}{3} + 2^{n-1-j}} - 2^{n-1-j}\lambda_1 + 2^{n-1}\lambda_2 = 0\]
Rearranging terms and dividing out by $2^{n-1-j}$, we see that
\[\lambda_1 - 2^j\lambda_2 = 2^{n-1} - \frac{2^{n-1+j} - 2^j}{3}\]
Now we can let $\lambda_1 = 2^{n-1} \leq M$ and $\lambda_2 = \frac{2^{n-1}-1}{3} \leq M.$ Furthermore, since $n$ is odd, and even powers of 2 must be 1 mod 3, $\lambda_2$ is indeed a positive integer. \\

Finally, we prove the equations are not $n$-regular. Note that $O_2(a_i) = i-1.$ So the $2$-adic orders of the coefficients in \eqref{eq5} are $0, 1, ..., n-1,$ and are thus all unique mod $n$. By Theorem 3.1, then, the equations \eqref{eq5} are not $n$-regular, regardless of the value of $n$.
\end{proof}

\noindent In 2005, Fox and Radoicic [5] showed that the family of equations 
\[x_1 + 2x_2 + ... + 2^{n-2}x_{n-1} - 2^{n-1}x_n \tag{9} \label{eq9} = 0\]
was not n-regular, and more recently Golowich [4] showed that the equations \eqref{eq9} are indeed $(n-1)$-regular.
Now we will generalize both of their findings, showing that, for any prime positive integer $p \geq 2$, the equation
\[x_1 + px_2 + ... + p^{n-2}x_{n-1} - p^{n-1}x_n = 0 \tag{10} \label{eq10}\]
has degree of regularity $n-1.$
\begin{theorem}
For any prime positive integer $p$, the family of equations
\[x_1 + px_2 + ... + p^{n-2}x_{n-1} - p^{n-1}x_n = 0\]
is $(n-1)$-regular but not $n$-regular.
\end{theorem}
\begin{proof}
Again, we begin by showing that the equation is always $(n-1)$-regular. If we take any $(n-1)$-coloring $C : \mathbb{N} \rightarrow [1,n-1] $, then there must be an $i,j$ in the set of $n$ integers $\{p^0, p^1, ..., p^{n-1}\}$ such that $C(p^i) = C(p^j)$ by the pigeonhole principle. Thus, there must be positive integers $x \in [1, p^{n-1}]$ and $y \in [1, n-1]$, such that $C(x) = C(p^yx)$. So we define a set $P$ as follows:
\[P = \{(\alpha,\beta) \in \mathbb{N}^2 \ \ : \ \ \alpha = p^j\beta \textnormal{ where } j \in [1, n-1] \}.\]
By Lemma 2.1, there is a  positive integer $d \in \mathbb{N}$ and set $\{p^j\beta,\beta\} \in P$, where $\beta \in [1, n-1]$, such that

\[p^j\beta + l_1d \ \ \kern 0.072em : \ |l_1| \leq M\]
\[\beta + l_2d \ \ \ \ \ : \ |l_2| \leq M\]
\[p^{n-1}d\]
are the same color. Thus we parametrize $x_1, ..., x_n$ as follows:
\[x_i = \left\{ \begin{array}{ll} p^{n-1}d \ \ \ \ \kern 0.014em \ \ \ : \ 1 \leq i < n, i \neq n-j \\ p^jb + l_1d \, \kern 0.1em \, \, \, \, \, : \ i = n-j \\ b + l_2d \kern 2.02em : \ i = n \end{array} \right.\]
where $l_1$ and $l_2$ are integers with $|l_1|, |l_2| \leq M$. Since all of $x_1, ..., x_n$ are in one of the 3 forms found through Lemma 2.1, they are all the same color. Since they must also satisfy \eqref{eq10}, we have that
\[p^{n-1}d (1 + p + ... + p^{n-2} - p^{n-1-j}) + p^{n-1-j}(p^j\beta + l_1d) - p^{n-1}(\beta + l_2d)\]
Simplifying the geometric sum in the first term and canceling out the $p^{n-1}$ and $-p^{n-1}$, we are left with
\[p^{n-1}\biggg{\frac{p^{n-1}-1}{p-1} - p^{n-1-j}} + p^{n-1-j}l_1 - p^{n-1}l_2 = 0\]
We can divide out by $p^{n-1-j}$ and rearrange terms to see that
\[l_1 - p^jl_2 = p^{n-1} - p^j\biggg{\frac{p^{n-1}-1}{p-1}}\]
Now we can let $l_1 = p^{n-1} \leq M$ and $l_2 = \frac{p^{n-1}-1}{p-1} = 1 + p + ... + p^{n-2} \leq M$. Therefore, regardless of the coloring, $l_1$ and $l_2$ produce a monochromatic solution $x_1, ..., x_n$ to \eqref{eq10}. \\ 
Again, we now show that the equations \eqref{eq10} are not $n$-regular. We proceed identically as in our previous proof, noting that $O_p(a_i) = i-1$ for all $i \leq n$, and thus all $O_p(a_i)$ are distinct mod $n.$ By Theorem 3.1, the family of equations \eqref{eq10} is not $n$-regular.
\end{proof}
Unfortunately, while the proof of $(n-1)$-regularity extends beyond just primes to all positive integers $a \geq 2$, the notion of $p$-adic order required to show that the equations are not $n$-regular does not extend easily beyond primes. For example, property 1 of $p$-adic order no longer holds. Given 2 numbers $q_1a^{k_1}$ and $q_2a^{k_2}$ with $a$-adic order $k_1$ and $k_2$, their product $q_1q_2a^{k_1 + k_2}$ may have order greater than $k_1 + k_2$ if $a \, | \, q_1q_2$.

Now that we have looked at specific equations in $n$ variables with degree of regularity $n-1$, we will present a condition which is sufficient to guarantee that a linear homogeneous equation in $n$ variables is $(n-1)$-regular.

\begin{theorem}
For any linear homogeneous equation $S$ in $n$ variables, if it can be expressed as 
\[\sum\limits_{1}^{m}a_ix_i - \sum\limits_{1}^{n-m}b_ix_i = 0 \tag{11} \label{eq11}\]
where all of $a_1, ..., a_m$ and $b_1, ..., b_{n-m}$ are positive integers, such that $\sum a_i ~ | \sum b_i$, then $S$ is $(n-1)$-regular.
\end{theorem}

\begin{proof}
We consider any $(n-1)$-coloring of the positive integers. Assume that $\sum a_i \ | \sum b_i$, and let $M = \sum a_i + \sum b_i.$ By Lemma 2.2, there are $k,d > 0$ such that
\[\sum b_i + \lambda d \ \ : \ \ |\lambda| \leq M \]
and
\[\sum a_i + d\]
are all the same color.
Now we parametrize $x_1, ..., x_n$ as follows:
\[   x_i = \left\{ \begin{array}{ll} \sum b_i + \lambda d \ \ \ \ : \ 1 \leq i \leq m \\ \sum a_i + d \ \ \ \ \ \kern 0.15em : \ m + 1 \leq i \leq n \end{array} \right.\]
In order for \eqref{eq11} to be satisfied, we must have
\[\sum a_i \biggg{\sum b_i + \lambda d} - \sum b_i\biggg{\sum a_i + d} = 0\]
Rearranging and solving for $\lambda$, we get that
\[\lambda = \frac{\sum b_i}{\sum a_i}\]
which is an integer since $\sum a_i \ | \sum b_i$. Note that, by symmetry, $\sum b_i \ | \sum a_i$ is also sufficient.
\end{proof}

\section{Conclusion}
In addition to having a way of generating equations in $n$ variables which are not $n$-regular through Theorem 3.1, Theorem 3.4 now gives us a way of generating equations which we know to be $(n-1)$-regular. Thus, using both, we can systematically create equations of any degree of regularity by guaranteeing that both conditions are satisfied. For example, consider the following equation
\[q_1x_1 + q_2px_2 + ... + q_{n-1}p^{n-2}x_{n-1} - (q_1 + q_2p + ... + q_{n-1}p^{n-2})p^{n-1}x_n = 0 \tag{12} \label{eq12}\]
where $p \nmid q_i$ for $1 \leq i \leq n$. By property 2 of $p$-adic order, $q_1 + ... + q_{n-1}p^{n-2}$ has $p$-adic order 0, and thus the equation is not $n$-regular by Theorem 3.1. Additionally, the sum of the positive coefficients divide the negative coefficient with quotient $p^{n-1}$, so the equation is $(n-1)$-regular by Theorem 3.4.

Not only does the family of equations \eqref{eq12} provide an alternate proof of Rado's Conjecture, but due to the looseness of restrictions on the values of $p$ and $q_1, ..., q_n$, given any positive integer $r$, it is immediately easy to find several different families of equations with degree of regularity $r$. The results in this paper therefore move us a lot closer to understanding what about an equation causes it to possess a certain degree of regularity; in other words, we may eventually see degree of regularity as a fundamental property of linear homogeneous equations. 
These results also present development and potential results in coding and information theory, which are focused on the secure and rapid transmission of information across information channels. In particular, a better understanding of the degree of regularity of linear equations should help us understand how certain patterns within the input set of an information channel affect the maximum rate at which a channel can communicate messages without error, referred to in information theory as the Shannon capacity [6].

However, current methods of examining the degree of regularity are still relatively weak and unclear. In order to reach a point where we can effectively talk about equations in terms of their degree of regularity, a useful next step would be to extend recent results in order to establish a necessary and sufficient condition for a linear homogeneous equation in $n$ variables to have degree of regularity $n-1.$ At the very least, further research in this area should move us toward classifying equations by their degree of regularity.
 
\end{document}